%% file: main.tex
\documentclass{amsart}
\usepackage{amssymb, mathtools, amsthm}
\usepackage{graphicx}
\usepackage{hyperref}
\usepackage{microtype}

\newtheorem{theorem}{Theorem}[section]
\newtheorem{bigthm}{Theorem}

\newtheorem{corollary}[theorem]{Corollary}
\newtheorem{proposition}[theorem]{Proposition}
\newtheorem{lemma}[theorem]{Lemma}
\newtheorem{definition}[theorem]{Definition}
\theoremstyle{remark}

\newtheorem{remark}[theorem]{Remark}

\DeclareMathOperator{\rank}{\mathrm{rank}}
\DeclareMathOperator{\im}{\mathrm{im}}
\usepackage[usenames,dvipsnames]{color}
 
\newcommand{\qbin}[2]{\begin{bmatrix}{#1}\\ {#2}\end{bmatrix}_q}
\newcommand{\Fq}{\mathbb{F}_q}

\newcommand{\Z}{\mathbb{Z}}
\newcommand\Fqn[1]{\mathbb{F}_q^{#1}}
\newcommand{\nul}{\mathrm{nul}}
\newcommand{\bP}{\mathbb{P}}

\newcommand{\ra}{\rightarrow}
\newcommand{\lra}{\longrightarrow}
\newcommand{\st}{\,|\,} 

\title{A Model for Random Chain Complexes}
\author{Michael J. Catanzaro and Matthew J. Zabka}
\date{\today}
\begin{document}
\begin{abstract}
We introduce a model for random chain complexes over a finite
field. The randomness in our complex comes from choosing the entries in the
matrices that represent the boundary maps uniformly over $\mathbb{F}_q$,
conditioned on ensuring that the composition of consecutive boundary maps is
the zero map.  We then investigate the combinatorial and homological 
properties of this random chain complex.
\end{abstract}
\maketitle

\section{Introduction}
\input{Intro.tex}
\section{Preliminaries}
\input{Preliminaries.tex}
\section{The Homology of a Random Chain Complex}\label{SecCondComp}
\input{Rand_Complex.tex}
\section{Proof of Theorem B}
\input{qtoinfty.tex}


\bibliography{master}
\bibliographystyle{plain}

\end{document}

%% file: Intro.tex
There have been a variety of attempts to randomize topological constructions.
Most famously, Erd\"os and R\'enyi introduced a model for random
graphs~\cite{erdos_random_1959}.  This work spawned an
entire industry of probabilistic models and tools used for understanding other
random topological and algebraic phenomenon. These include various models for
random simplicial complexes, random networks, and many
more~\cite{erdos_evolution_1960, linial_random_2017}. Further, this has
led to beautiful connections with statistical physics, for example through
percolation theory~\cite{bollobas_2006_percolation, broadbent_percolation_1957,
kesten_percolation_1982}.  Our ultimate goal is to understand higher
dimensional topological constructions arising in algebraic topology from a
random perspective. In this manuscript, we begin to address this goal with the
much simpler objective of understanding an algebraic construction commonly associated
with topological spaces, known as a chain complex. 

Chain complexes are used to measure a
variety of different algebraic, geoemtric, and topological properties Their usefulness lies in
providing a pathway for homological algebra computations. They arise in a
variety of contexts, including commutative algebra, algebraic geometry, group
cohomology, Hoschild homology, de Rham cohomology, and of course algebraic
topology~\cite{bott2013differential, brown2012cohomology,
hartshorne2013algebraic, hatcher2002algebraic, hochschild1945cohomology}.
Specifically, chain complexes measure the relationship between cycles and
boundaries of a topological space. This relationship uncovers many topological
properties of interest, and is precisely what homology reveals. Furthermore,
the singular chain complex of a topological space provides a canonical method
of associating a chain complex to a topological space.

%

Let $R$ be a ring. A {\em chain complex $C_*=(C_m, \delta_m)$ with coefficients in
$R$} is a sequence of $R$-modules, denoted $C_m$, together with a sequence of
linear transformations 
\[
  \cdots \xrightarrow{\delta_{m+1}} C_m \xrightarrow{\delta_m}
  C_{m-1} \xrightarrow{\delta_{m-1}} \cdots
\]
such that $\delta_{m-1}\delta_m = 0$ for all $m \in \Z$.  The maps $\delta_m$
are called the boundary maps of the chain complex, and the equation
$\delta_{m-1} \delta_m = 0$ is known as the boundary condition;
see~\cite{cartan2016homological} for further details. 

The boundary condition $\delta_{m-1}\delta_m=0$ forces $\im \delta_m \subseteq \ker \delta_{m-1}$.
The {\em homology} of a chain complex measures the deviation of this containment
from equality:
\begin{equation*}
  \label{eqn:hom}
  H_m(C_*;R) = \frac{\ker \delta_{m-1}}{\im \delta_m} \, .
\end{equation*}
When the chain complex arises by taking singular chains on a topological
space, homology can be a very powerful tool in algebraic topology~\cite{hatcher2002algebraic}.

We work over the field with $q$-elements $R = \Fqn{}$ and consider the chain
complex whose $R$-modules are given by finite dimensional vector spaces,  $C_m
= \Fqn{n_m}$, where each $n_m \in \mathbb{N}$. After fixing the standard
basis for $\Fqn{}$, the boundary
maps can be regarded as $n_{m-1}\times n_m $ matrices, which we denote by
$A_m$. Homology can then be understood in terms of dimension
\begin{equation*}
  \beta_m := \dim_{\Fqn{}} \frac{\ker A_{m-1}}{\im A_m} \, ,
  \label{eqn:betti_numbers}
\end{equation*}
where $\beta_m$ is known as the $m^\mathrm{th}$ {\em Betti number}.

%


\subsection*{Main Results}
Let $q$ be a prime number.  We
build a random chain complex with coefficients in $\Fqn{}$ as follows (see
Definition~\ref{defn:random_chain_cx} for a precise statement). Given a
sequence of non-negative integers $\{n_m\}$, where $m \in \Z$, together with a probability
distribution $\varphi$ on $\Fqn{}$, we construct random
linear transformations 
\[
  A_m : \Fqn{n_m} \lra \Fqn{n_{m-1}} \, ,
\]
for all $m$. The transformations are subject to the constraint $A_{m-1} A_m =
0$, and should be chosen according to $\varphi$. The latter means the
following: After fixing the standard basis for $\Fqn{n_m}$, it suffices to
construct random $n_{m-1} \times n_m$ matrices $A_m$, satisfying $A_{m-1}A_m =
0$. We do so by choosing matrix entries i.i.d. from the distribution $\varphi$
on $\Fqn{}$. We then say that $(\Fqn{n_m}, A_m, \varphi)$ is a {\em random
chain complex}. 

We are primarily interested in the case when $\varphi$ is the discrete uniform
distribution on $\Fqn{}$. In this case, we drop $\varphi$ from the notation and
say that $(\Fqn{n_m}, A_m)$ is a {\em uniform random chain complex}. We also
restrict our attention to bounded below chain complexes (see Remark~\ref{rem:bdd}).

Our first result is an explicit formula for the distribution
of the Betti numbers.
\begin{bigthm} 
  \label{thm:bettinum}
  Let $\beta_m$ be the $m$-th Betti number of a uniform random chain complex
  $(\Fq^{n_m} , A_m)$. Then
  \[    
    \bP[\beta_m = b] = \sum_{i_m=0}^{n_{m}} P_{i_m}^{m}(i_m-b)
    \sum_{i_{m-1}=0}^{n_{m-1}} P_{i_{m-1}}^{m-1}\left(n_{m} -i_m\right) \cdots
    \sum_{i_1 = 0}^{n_1} P_{i_1}^1\left(n_2 - i_2\right) P_{n_0}^0 \left(n_1 - i_1\right) \, ,
  \]
  where $P^m_k(r)$ is given in Eq.~\eqref{eqn:Pmkr}.
\end{bigthm}

As Theorem~\ref{thm:bettinum} gives a formula for
computing the distribution of the Betti numbers, it also leads to formulas for
other probabilistic properties of $\beta_m$, such as its moments and variance.

Our second main result show that, asymptotically, the $m$-th Betti number of a uniform random chain complex concentrates in a single value. Set
\begin{equation*}
  (n)_+ = \max(0,n) \, , 
\end{equation*}
to be the {\em positive part} of $n$.
Define
\begin{equation}
  B_m = (-n_{m+1} + (n_m - (n_{m-1} - (\cdots - (n_1 - n_0)_+ \cdots)_+ )_+
  )_+)_+ \, \, .
  \label{eqn:Nm}
\end{equation}

\begin{bigthm}
  \label{thm:qtoinfty}
  For a uniform random chain complex $(\Fqn{n_m},A_m)$ with $B_m$ defined as in Eq.~\eqref{eqn:Nm},
  \[
    \bP[\beta_m = B_m] \ra 1 
    \mbox{ as } q \ra \infty  \, .
  \]
\end{bigthm}


\begin{remark}
  \label{rem:monotone}
  As a special case of the above theorem, consider when $\{n_m\}$
  is constant or increasing. In this case, 
  $B_m = 0$, and the homology is trivial in probability as $q \ra \infty$ 
  (see Corollary~\ref{cor:inc}). 
\end{remark}

\subsection*{Related Work} Others have considered different methods of applying
randomness to chain complexes. In~\cite{ginzburg2017random}, Ginzburg and
Pasechnik investigate a different notion of a random chain complex than the one
we have described above.  Given a finite dimensional vector space $V$ over
$\Fqn{}$, they consider chain complexes of the form \[ \cdots
\stackrel{D}{\lra} V \stackrel{D}{\lra} V \stackrel{D}{\lra} \cdots \, \, , \]
for a randomly chosen linear operator $D$ such that $D^2 = 0$. They choose the
operator $D$ uniformly over all such possible choices. In particular, our
construction is distinct from theirs, since they use the same operator $D$ at
each stage of the complex.
The first of their main results~\cite[Thm 2.1]{ginzburg2017random} states that
the rank of homology concentrates in the lowest possible dimension as $q \ra \infty$.  
In the special case when $n_m \equiv n$ is constant, we also obtain minimal rank
homology (see Remark~\ref{rem:monotone}).


The second author has introduced and studied the
properties of a random Bockstein operation~\cite{zabka2018random}. 
In homological algebra,
the Bockstein is a connecting homomorphism associated with a short exact
sequence of abelian groups, which are then used as the coefficients in a chain
complex. Given a random boundary operator of a chain complex, the distribution of compatible
random Bockstein operations is given in~\cite[Thm 5.2]{zabka2018random}.

\subsection*{Outline}
  The paper is organized as follows. In Section 2, we discuss preliminary results useful for 
  the combinatorial aspects of our results. We give a precise definition of a model for a 
  random chain complex in Section 3, as well as prove
  Theorem~\ref{thm:bettinum}. In Section 4, we complete the proof of Theorem~\ref{thm:qtoinfty}.


  \subsection*{Acknowledgments} The first author would like to thank Peter Bubenik for helpful discussions.

%% file: Preliminaries.tex
This section consists of lemmas that are necessary to prove our main results.
The first four lemmas count the number of elements in various sets related to
finite vector spaces over $\mathbb{F}_q$.  We provide
proofs for these lemmas, but an interested reader can also
see~\cite{stanley2011enumerative} for further details. The last lemma of 
this section, Lemma~\ref{lem:Pqtoinfty}, gives the asymptotic behavior of a useful
conditional probability and will be used several times in the remainder of the
paper.

\begin{lemma}\label{NumkTup}
  The number of ordered, linearly independent $k$-tuples of vectors in $\Fqn{n}$  is
\[
\prod_{j=0}^{k-1} \left(q^n - q^j\right) = 
(q^n-1)(q^n-q)(q^n-q^2)\cdots(q^n - q^{k-2})(q^n-q^{k-1}).
\]
\end{lemma}

\begin{proof}
Since first vector in the $k$-tuple may be any vector except for the zero
vector, there are $q^n-1$ choices for the first vector.  More generally, for $1
\leq m \leq k$, the
$m$-th vector in the $k$-tuple may be any vector that is not a linear
combination of the previously chosen $m-1$ vectors. So there are $q^n -
q^{m-1}$ choices for the $m$-th vector.
\end{proof}

\begin{lemma}\label{NumkSub}
  The number of $k$-dimensional subspaces of $\Fqn{n}$ is
\[
  \qbin{n}{k} = \prod_{j=0}^{k-1} \frac{q^n-q^j}{q^k-q^j} \, .
\]
\end{lemma}

\begin{proof}
  Let $\left[\begin{smallmatrix}n\\k\end{smallmatrix}\right]_q$ denote the number of $k$-dimensional subspaces of $\Fqn{n}$
  and $N(q,k)$ be the number of ordered, linear independent $k$-tuples of
  vectors in $\Fqn{n}$.  Then Lemma~\ref{NumkTup} gives
  \begin{equation}
    \label{Nqk1}
    N(q,k) = \prod_{j=0}^{k-1} q^n-q^j \, .
  \end{equation}
  We may also find $N(q,k)$ another way: First choose a $k$-dimensional
  subspace and then choose the independent vectors in our $k$-tuple from the
  chosen subspace.  There are
  $\left[\begin{smallmatrix}n\\k\end{smallmatrix}\right]_q$ $k$-dimensional
  subspaces of $\Fqn{n}$.  There are $q^k-1$ choices for the first vector in
  the $k$-tuple, and more generally, for $1 \leq m \leq k$, there are $q^k - q^{m-1}$ vectors for
  the $m$-th vector in the $k$-tuple.  Thus
  \begin{equation}\label{Nqk2}
    N(q,k) = \qbin{n}{k}\prod_{j=0}^{k-1} q^k - q^j \, .
  \end{equation}
  Equations~\eqref{Nqk1}~and~\eqref{Nqk2} give the desired result.
\end{proof}

The number $\left[\begin{smallmatrix}n\\k\end{smallmatrix}\right]_q$ defined
above is known as the $q$-binomial
coefficient~\cite{stanley2011enumerative}.  
Lemmas~\ref{NumkTup}~and~\ref{NumkSub} combine to count the number of matrices
of a given rank. 

\begin{lemma}\label{Num_mbyn_rankr}
The number of $m\times n$ matrices of rank $r$ with entries in $\Fq$ is given by
\begin{equation*}
      \prod_{j=0}^{r-1} \frac{(q^m-q^j) (q^n - q^j)}{q^r - q^j} \, .
\end{equation*}
\end{lemma}

\begin{proof}
  Let $W$ be a fixed $r$-dimensional subspace of $\Fqn{n}$.  The number of
  matrices whose column space is $W$ is given by the number of $r\times n$
  matrices with rank $r$.  This number is given by Lemma \ref{NumkTup}. The
  number of $r$-dimensional subspaces of $\Fqn{n}$ is
  $\left[\begin{smallmatrix}m\\r\end{smallmatrix}\right]_q$,  as stated in
  Lemma \ref{NumkSub}.  The product of these is the number of $m\times n$
  rank $r$ matrices.
\end{proof}


\begin{definition}\label{defPkmr}
Let $n_m$ be a sequence of natural numbers. Let $A_m$ be a sequence of random
$(n_m) \times (n_{m-1})$ matrices whose entries are chosen i.i.d. uniformly
from $\Fq$. Let $r$ be a non-negative integer.  Define 
\[
  P^m_k(r) := \mathbb{P} 
  \left[\rank(A_{m+1}) = r \st A_{m}A_{m+1} = 0, \nul(A_{m}) = k \right] \, .
\]
\end{definition}


\begin{lemma}\label{lemPkmr} With $A_m$ defined as in Definition \ref{defPkmr}, we have that
  \begin{equation}
    P^m_k(r) = 
    q^{-kn_{m+1}}\prod_{j=0}^{r-1} \frac{(q^{n_{m+1}}-q^j) (q^k - q^j)}{q^r - q^j} \, . 
	 \label{eqn:Pmkr}
       \end{equation}
\end{lemma}
\begin{proof}
  Let $ k = \nul(A_m)$.
The linear transformation $A_{m+1}$ maps 
$\Fq^{n_{m+1}}$ into a $k$-dimensional subspace of
$\Fq^{n_{m}}$. By changing basis, $A_{m+1}$ can be represented by an
$k \times n_{m+1}$ matrix. There are $q^{kn_{m+1}}$ total $k \times n_{m+1}$
matrices over $\Fqn{}$, and by
Lemma \ref{Num_mbyn_rankr}, there are 
\[
  \prod_{j=0}^{r-1} \frac{(q^{n_{m+1}}-q^j) ( q^k - q^j)}{q^r-q^j}
\]
such matrices of rank $r$.
\end{proof}

\begin{remark}
  We adopt the convention that the empty product is 1. With this,
  $P^m_0(0) = 1$ and $P^m_k(r) = 0$ for impossible cases like $r>k$ and $k < 0$.
\end{remark}

\begin{lemma}
  \label{lem:Pqtoinfty}
  Fix $m$ and $k$. Then 
  \[
    \lim_{q \ra \infty} P^m_k(r) = 
      \begin{cases}
        1 & \mbox{ if } r = \min(k,n_{m+1}) \, , \\
        0 & \mbox { else.}
      \end{cases}
    \]
\end{lemma}

\begin{proof}
  Suppose $\min(k,n_{m+1}) = k$. Then 
  \begin{align*} 
    P^m_k(k) =& q^{-kn_{m+1}}\prod_{j=0}^{k-1} \frac{(q^{n_m+1}-q^j) (q^k - q^j)}{q^k - q^j}  \\
    =& q^{-kn_{m+1}} \prod_{j=0}^{k-1}(q^{n_{m+1}} - q^j) \\
    =& \prod_{j=0}^{k-1} (1-q^{j-n_{m+1}}).
  \end{align*}
  Suppose $\min(k,n_{m+1}) = n_{m+1}$. Then
  \begin{align*}
    P^m_k(n_{m+1}) 
    =& q^{-kn_{m+1}}\prod_{j=0}^{n_{m+1}-1} \frac{(q^{n_{m+1}} - q^j)(q^k - q^j)}{q^{n_{m+1}} - q^j}    \\
    =&  \prod_{j=0}^{n_{m+1}-1}(1-q^{j-k}).
	\end{align*}
    In either of the above cases, $P^m_k(r) \ra 1$ as $q \ra \infty$. On the other hand, 
    if $r \neq \min(k,n_{m+1})$, then $P^m_k(r) \ra 0$ since each $P^m_k(r)$ represents
    a probability by Definition~\ref{defPkmr}.
\end{proof}

%% file: Rand_Complex.tex



\begin{definition}
  Let $q$ be a prime number and $\{n_m\}$ be a sequence of non-negative integers indexed by $m \in \Z$. Let $\varphi$ be a probability distribution on $\Fqn{}$. Let $\{A_m\}$ be a sequence of $n_{m-1}\times n_{m}$ matrices whose entries are chosen i.i.d. according to $\varphi$, subject to the condition that $A_{m-1}A_m = 0$.  The triple $(\Fqn{n_m},A_m,\varphi)$ is then said to be a {\bf  model for a random chain complex over} the field $\Fqn{}$.
\end{definition}


\begin{definition}
  A {\bf uniform random chain complex} is a model for a random chain complex over $\Fqn{}$, $(\Fqn{n_m},A_m,\varphi)$, where $\varphi$ is the uniform distribution on $\Fqn{}$. In this case, we drop $\varphi$ from the notation and write $(\Fqn{n_m}, A_m)$.
  \label{defn:random_chain_cx}
\end{definition}

\begin{remark}
  \label{rem:bdd}
  We are interested in {\em bounded from below} chain complexes, so that $A_m = 0$ for all $m < 0$
  for the remainder of the manuscript.
\end{remark}


We wish to investigate the probabilistic properties of the homology of a uniform random
chain complex.  We are primarily interested in the distribution of the Betti
numbers $\beta_m = \nul \,A_m - \rank A_{m+1}$.

\begin{remark}
If $\{A_m\}$ is the sequence of maps from a uniform random chain complex, Definition
\ref{defPkmr} immediately gives us that
\[
  P^m_k(r) = \mathbb{P}\left[\beta_m = k-r \st \nul(A_{m}) = k\right].
\]
\end{remark}

\begin{theorem}\label{thmProbranks}
Let $(\Fq^{n_m}, A_m)$ be a uniform random chain complex 
and $A_0:\Fq^{n_0}\to 0$.  Then
\begin{align*}
 &  \mathbb{P}\left[\rank(A_{m}) = n_{m} - k\right]\\
=& 	\sum_{i_{m-1}=0}^{n_{m-1}} P_{i_{m-1}}^{m-1}\left(n_{m} -k\right)
	\sum_{i_{m-2}=0}^{n_{m-2}} P_{i_{m-2}}^{m-2}\left(n_{m-1} - i_{m-1}\right)
		\cdots
        \sum_{i_1 = 0}^{n_1} P_{i_1}^1\left(n_2 - i_2\right) P_{n_0}^0 \left(n_1 - i_1\right) \, .
\end{align*}
\end{theorem}
\begin{proof}
The proof is by induction on $m$. For the base case $m=1$, we have
\begin{align*}
\mathbb{P}\left[\rank(A_1) = n_1 - k\right]
	=& \sum_{i_0=0}^{n_0}\mathbb{P}\left [\rank(A_1)= n_1 - k\st\nul (A_0) = i_0\right]
		\mathbb{P}\left[\nul(A_0) = i_0 \right]\\
	=& \mathbb{P}\left [\rank(A_1)= n_1 - k\st\nul (A_0) = n_0\right]\\
	=& P_{n_0}^0(n_1 - k).
\end{align*}
The first equality follows by the Law of Total Probability, and the second equality follows because $A_0$ is the zero map.

For the inductive step, suppose that
\begin{align*}
 &  \mathbb{P}\left[\rank(A_{m-1}) = n_{m-1} - i_{m-1}\right]\\
=& 	\sum_{i_{m-2}=0}^{n_{m-2}} P_{i_{m-2}}^{m-2}\left(n_{m-1} -i_{m-1}\right)
	\sum_{i_{m-3}=0}^{n_{m-3}} P_{i_{m-3}}^{m-3}\left(n_{m-2} - i_{m-2}\right)
		\cdots
	\sum_{i_1 = 0}^{n_1} P_{i_1}^1\left(n_2 - i_2\right) P_{n_0}^0 \left(n_1 - i_1\right).
\end{align*}
As in the base case, we have 
\begin{align*}
   & \mathbb{P}\left[\rank(A_{m}) = n_{m} - k\right]\\
  =& \sum_{i_{m-1}=0}^{n_{m-1}} \mathbb{P}\left[ \rank(A_{m}) = n_{m} - k  
  		\st \nul(A_{m-1}) = i_{m-1}\right] \mathbb{P}\left[\nul(A_{m-1}) = i_{m-1}\right]\\
  =& \sum_{i_{m-1}=0}^{n_{m-1}} P_{i_{m-1}}^{m-1}(n_{m} - k)
  		\mathbb{P}\left[\rank(A_{m-1}) = n_{m-1}- i_{m-1}\right].
\end{align*}
The desired result now follows by the induction hypothesis.
\end{proof}

Theorem~\ref{thm:bettinum} now follows from Theorem~\ref{thmProbranks} and 
Lemma~\ref{lem:Pqtoinfty} in a straightforward manner. We give an explicit proof
for completeness.


\begin{proof}[Proof of Theorem~\ref{thm:bettinum}]
  By the law of total probability, we have
  \begin{align*}
    \mathbb{P}[\beta_{m} = b] 
    &= \mathbb{P}[\rank(A_{m+1}) = \nul(A_{m}) - b]\\
    &= \sum_{k=0}^{n_{m}}\mathbb{P}[\rank(A_{m+1}) = k - b\st \nul(A_{m})=k] \mathbb{P}[\nul(A_{m}) = k]\\
    &= \sum_{k=0}^{n_{m}} P_{k}^{m}(k-b)\mathbb{P}[n_{m} - \rank(A_{m}) = k]\\
    &= \sum_{k=0}^{n_{m}} P_{k}^{m}(k-b)\mathbb{P}[\rank(A_{m}) = n_{m} - k] \, .
  \end{align*}
  By Theorem~\ref{thmProbranks},
  \begin{align*}
    & \sum_{k=0}^{n_{m}} P_{k}^{m}(k-b)\mathbb{P}[\rank(A_{m}) = n_{m} - k] \\
    =& \sum_{k=0}^{n_{m}} P_{k}^{m}(k-b)\sum_{i_{m-1}=0}^{n_{m-1}} P_{i_{m-1}}^{m-1}\left(n_{m} -k\right)
    \cdots
    \sum_{i_1 = 0}^{n_1} P_{i_1}^1\left(n_2 - i_2\right) P_{n_0}^0 \left(n_1 - i_1\right), 
  \end{align*}
  as desired.
\end{proof}

%% file: qtoinfty.tex
In this section, we analyze Theorem~\ref{thm:bettinum} under the 
limit $q \ra \infty$.

\begin{proposition}
  \label{prop:oneseq}
Let $I_m:= \{0,1,\ldots, n_m\}$ and let $I^{(j)}:= I_1\times\cdots \times I_j$.
Then for every natural number $j$, there exists exactly one $\mathbf{i}^\ast =
(i_1^\ast,\ldots, i_j^\ast)$ in $I^{(j)}$ such that 
\[
  P_{i_{j-1}^\ast}^{j-1}(n_j-i_j^\ast)\cdots
  P_{i_1^\ast}^1(n_2-i_2^\ast)P_{n_0}^0(n_1-i_1^\ast) \to 1 \, ,
\]
as $q\to\infty$. In particular, set $i_0^\ast = n_0$. Then for $\ell$ in
$\{1,2,\ldots, j\}$, we have $i_\ell^\ast = (n_\ell - i_{\ell - 1}^\ast)_+$.
\end{proposition}

\begin{proof}
The proof is by induction on $j$.

Base step ($j=1$). By Lemma~\ref{lem:Pqtoinfty}, we have $P_{n_0}^0 (n_1 -
i_1^\ast) \to 1$ as $q\to\infty$ if and only if $n_1 - i_1^\ast = \min(n_0,
n_1)$.  That is, $i_1^\ast = (n_1 - n_o)_+ = (n_1 - i_0^\ast)_+$.

Inductive step. Assume there exists exactly one $(i_1^\ast,\ldots ,
i_{j-1}^\ast)$ in $I^{(j-1)}$, with $i_\ell = (n_\ell - i_{\ell-1}^\ast)_+$ for
$\ell$ in $\{1,2,\ldots, j-1\}$, such that 
\[
P_{i_{j-2}^\ast}^{j-2}(n_{j-1} - i_{j-1}^\ast) 
\cdots P_{i_1^\ast}^1 (n_2 - i_2^\ast) P_{n_0}^0 (n_1 - i_1^\ast) \to 1 \, ,
\]
as $q\to\infty$.  By Lemma~\ref{lem:Pqtoinfty}, $P_{i_{j-1}^\ast}^{j-1}(n_j -
i_j^\ast) \to 1$ as $q \to \infty$ if and only if $n_j - i_j^\ast =
\min(i_{j-1}^\ast, n_j)$. That is, $i_j^\ast = (n_j - i_{j-1}^\ast)_+$.  
For $\mathbf{i} = (i_1^\ast,\ldots, i_{j-1}^\ast, i_j^\ast)$ in $I^{(j)}$, we
have
\[
P_{i_{j-1}^\ast}^{j-1}(n_j-i_j^\ast)P_{i_{j-2}^\ast}^{j-2}(n_{j-1}-i_{j-1}^\ast)\cdots P_{i_1^\ast}^1(n_2-i_2^\ast)P_{n_0}^0(n_1-i_1^\ast) \to 1
\]
as $q\to\infty$, as desired.
\end{proof}

\begin{proof}[Proof of Theorem~\ref{thm:qtoinfty}]
  By Theorem~\ref{thm:bettinum}, it is sufficient to show
  \[
    P^m_{i_m^\ast}(i_m^\ast-b) P_{i^\ast_{m-1}}^{m-1}\left(n_{m} -i_m^\ast\right) \cdots
    P_{i_1^\ast}^1\left(n_2 - i_2^\ast\right) P_{n_0}^0 \left(n_1 - i_1^\ast\right) \to 1 
  \]
  as $q \ra \infty$ for a single sequence $\mathbf{i}^\ast=(i_0^\ast, \ldots,
  i_m^\ast)$ and a single
  value of $b$. After choosing $\mathbf{i}^\ast$ as in Proposition~\ref{prop:oneseq},
  the value of $b$ is easily determined from Lemma~\ref{lem:Pqtoinfty} to be
	\begin{align*}
	b 	&= i_m^\ast - \min(i_m^\ast, n_{m+1})\\
		&= (-n_{m+1} + i_m^\ast)_+\\
		&= (-n_{m+1} + (n_m -i_{m-1}^\ast)_+)_+\\
		&= (-n_{m+1} + (n_m - (n_{m-1} - (\cdots (n_1 - n_0)_+ \cdots)_+)_+)_+ )_+\\
        &= B_m \, . \qedhere
	\end{align*}
\end{proof}

Proposition~\ref{prop:oneseq} and Theorem B have a number of immediate consequences.

\begin{corollary}
  Let $(\Fqn{n_m},A_m)$ be a uniform random chain complex. Then
  \[
    \bP[\rank(A_m) = n_m-(n_m-(n_{m-1} -( \cdots -(n_1-n_0)_+ \cdots )_+ )_+)_+]
    \ra 1 \,
  \]
  as $q \ra \infty$.
\end{corollary}

\begin{proof}
  Using Lemma~\ref{lem:Pqtoinfty}, this follows by a similar argument to the
  Proof of Theorem~\ref{thm:qtoinfty}.
\end{proof}

\begin{corollary}
  \label{cor:inc}
  If $\{n_m\}$ is a monotone increasing sequence, then 
  \[
    \lim_{q \ra \infty} \bP[\beta_m = 0] = 1 \, .
  \]
\end{corollary}

\begin{proof}
  By direct inspection, we have
  \[
    (n_m - (n_{m-1} - ( \cdots (n_1 - n_0)_+ \cdots )_+)_+)_+ \leq n_m \, ,
  \]
  and hence $B_m = 0$.
\end{proof}

\begin{corollary}
  The $t$-th moments of the random variable $\beta_m$ satisfy
  \[
    \lim_{q \ra \infty} \mathbb{E}\left[\beta_m^t \right] = B_m^t \, .
  \]
\end{corollary}